\documentclass[a4paper, 10pt, conference]{ieeeconf}

\IEEEoverridecommandlockouts                              

\usepackage{graphicx}
\usepackage{amsmath}
\usepackage{amssymb}

\usepackage{dsfont}
\usepackage{caption}
\usepackage{subcaption}

\usepackage{amsthm}
\newtheorem{theorem}{Theorem}

\newtheorem{lemma}[theorem]{Lemma}

\newtheorem{corollary}[theorem]{Corollary}

\newtheorem{definition}{Definition}

\newtheorem{remark}{Remark}

\usepackage{xcolor}

\newcommand{\bal}[1] {\ensuremath{\left(\begin{array}{#1}}}
\newcommand{\ear} {\ensuremath{\end{array}\right)}}

\newcommand{\bals}[1] {\ensuremath{\left[\begin{array}{#1}}} 
\newcommand{\ears} {\ensuremath{\end{array} \right] }} 

\DeclareMathOperator{\trace}{tr}

\DeclareMathOperator{\diag}{diag}

\newcommand{\one} {\ensuremath{\mathds{1} }} 

\newcommand{\funcRdR}{\ensuremath{{f}}}
\newcommand{\funcRdRd}{\ensuremath{{f}}}

\let\leq\leqslant
\let\geq\geqslant

\newcommand{\calC}{\ensuremath{\mathcal{C}}}
\newcommand{\calD}{\ensuremath{\mathcal{D}}}

\newcommand{\calF}{\ensuremath{\mathcal{F}}}

\newcommand{\calL}{\ensuremath{\mathcal{L}}}

\newcommand{\calR}{\ensuremath{\mathcal{R}}}

\newcommand{\calT}{\ensuremath{\mathcal{T}}}

\newcommand{\calX}{\ensuremath{\mathcal{X}}}

\newcommand{\frakso}{\ensuremath{\mathfrak{so}}}

\newcommand{\bmat}{\begin{matrix}}
\newcommand{\emat}{\end{matrix}}
\newcommand{\bbm}{\begin{bmatrix}}
\newcommand{\ebm}{\end{bmatrix}}
\newcommand{\bpm}{\begin{pmatrix}}
\newcommand{\epm}{\end{pmatrix}}
\newcommand{\bse}{\begin{subequations}}
\newcommand{\ese}{\end{subequations}}
\newcommand{\beq}{\begin{equation}}
\newcommand{\eeq}{\end{equation}}
\newcommand{\ben}{\begin{enumerate}}
\newcommand{\een}{\end{enumerate}}
\newcommand{\beni}{\renewcommand{\labelenumi}{\roman{enumi}.}
\renewcommand{\theenumi}{\roman{enumi}}\begin{enumerate}}
\newcommand{\eeni}{\end{enumerate}\renewcommand{\labelenumi}{\arabic{enumi}.}
\renewcommand{\theenumi}{\arabic{enumi}}}
\newcommand{\bena}{\renewcommand{\labelenumi}{\alpha{enumi}.}
\renewcommand{\theenumi}{\alpha{enumi}}\begin{enumerate}}
\newcommand{\eena}{\end{enumerate}\renewcommand{\labelenumi}{\arabic{enumi}.}
\renewcommand{\theenumi}{\arabic{enumi}}}
\newcommand{\bit}{\begin{itemize}}
\newcommand{\eit}{\end{itemize}}

\newcommand{\R}{\ensuremath{\mathbb R}}

\title{\LARGE \bf
Asymptotic and finite-time almost global attitude tracking: representations free approach}

\author{Jieqiang Wei, Christos Verginis, Junfeng Wu, Dimos V. Dimarogonas, Henrik Sandberg and Karl H. Johansson  
\thanks{*This work is supported by Knut and Alice Wallenberg Foundation, Swedish Research Council, and Swedish Foundation for Strategic Research.}
\thanks{Jieqiang Wei, Christos Verginis, Dimos V. Dimarogonas, Henrik Sandberg, Karl H. Johansson are with the ACCESS Linnaeus Centre, School of Electrical Engineering. 
 KTH Royal Institute of Technology,
 SE-100 44 Stockholm, Sweden.
 {\tt\small \{jieqiang, verginis, dimos, hsan, kallej\}@kth.se}.
 \newline
 Junfeng Wu is with College of Control Science and Engineering, Zhejiang University, Hangzhou, China. 
 {\tt\small jfwu@zju.edu.cn}
}}

\begin{document}

\maketitle

\begin{abstract}\label{s:Abstract}
In this paper, the attitude tracking problem is considered using the rotation matrices. Due to the inherent topological restriction, it is impossible to achieve global attractivity with any continuous attitude control system on $SO(3)$. Hence in this work, we propose some control protocols achieve almost global tracking asymptotically and in finite time, respectively. In these protocols, no world frame is needed and only relative state informations are requested. For finite-time tracking case, Filippov solutions and non-smooth analysis techniques are adopted to handle the discontinuities. Simulation examples are provided to verify the performances of the control protocols designed in this paper.
\end{abstract}


\IEEEpeerreviewmaketitle

\section{Introduction}\label{s:Introduction}

Originally motivated by aerospace developments in the middle of the last century~\cite{Bower1964,Kowalik1970}, the rigid body attitude control problem has continued to attract attention with many applications such as aircraft attitude control \cite{Athanasopoulos2014,Tsiotras1994}, spacial grabbing technology of manipulators~\cite{ZXLi}, target surveillance by unmanned vehicles~\cite{pettersen1996position}, and camera calibration in computer vision~\cite{ma2012invitation}. Furthermore, the configuration space of rigid-body attitudes is the compact non-Euclidean manifold $SO(3)$, which poses theoretical challenges for attitude control \cite{Bhat00scl}. The coordination of multiple attitudes is of high interest both in academic and industrial research, e.g., \cite{dong2016attitude,sarlette2009autonomous,Thunberg2016auto}.

Here we review some related existing work. As attitude systems evolves on $SO(3)$---a compact manifold without a boundary---there exists no continuous control law that achieves global asymptotic stability \cite{brockett83asymptotic}. Hence one has to resort to some hybrid or discontinuous approaches. In \cite{lee2015global}, exponential stability is guaranteed for the tracking problem for a single attitude. 
In \cite{li2012decentralized} the authors considered the synchronization problem of attitudes under a leader-follower architecture. In \cite{pereira2016common}, the authors provided a local result on attitude synchronization. Based on a passivity approach, \cite{ren2010distributed} proposed a consensus control protocol for multiple rigid bodies with attitudes represented by modified Rodrigues parameters. In \cite{tron2012intrinsic}, the authors provided a control protocol in discrete time that achieves almost global synchronization, but it requires  global knowledge of the graph topology.
Although there exists no continuous control law that achieves global asymptotic stability, a methodology based on the axis-angle representation obtains almost global stability for attitude synchronization under directed and switching interconnection topologies is proposed in \cite{Thunberg2014auto}. 

Besides these agreement results, some tracking results are reviewed as follows.
In \cite{LEE2012}, an almost global attitude tracking control system based on an alternative attitude error function is proposed. This attitude error function is not differentiable at certain attitudes and employs the Frobenius attitude difference,
and the resulting control input is not continuous. In \cite{Lee2010}, one tracking protocol is proposed for unmanned aerial vehicle (UAV), again using Frobenius state differences. So far,  finite-time attitude tracking problems are studied in different settings, e.g., \cite{Du2011,Zong2016}. In \cite{Du2011}, finite-time attitude synchronization was investigated in a leader-follower architecture, namely all the followers tracking the attitude of the leader. In \cite{Zong2016}, quaternion representation was employed for finite-time attitude synchronization. Both works used continuous control protocols with high-gain.



In this paper, we shall focus on the attitude tracking problem, based on the rotation matrices in $SO(3)$. The contributions are threefolds. First, based on the two types of relative state difference, geodesic or Frobenius, two types of control schemes are proposed. Let us refer these two type of protocols as geodesic and Frobenius controller, respectively. In both types of the controllers, only the relative state informations, with no world frame, are needed. Second, for both geodesic and Frobenius controllers, we propose one for asymptotic tracking and one for finite-time tracking. More precisely, sign function is employed for the finite-time case. Since these control schemes are discontinuous,  nonsmooth analysis is employed throughout the paper. Third, all the controllers designed in this paper achieves almost global tracking.

The structure of the paper is as follows. In Section \ref{s:Preliminaries}, we review some results for the special orthogonal group $SO(3)$ and introduce some terminologies and notations in the context of discontinuous dynamical systems. Section \ref{ss:basic_model} presents the problem formulation of the attitude tracking. The main results of the stability analysis of the finite-time convergence are presented in Section~\ref{s:main}, where two types of controllers, using geodesic and Frobenius state differences, respectively, are proposed to achieve almost global tracking. The simulations of the main results are in Section \ref{s:simulation}. Then, in Section \ref{s:conclusion}, the paper is concluded.

\textbf{Notations.} With $\R_-,\R_+, \R_{\geq 0}$ and $\R_{\leqslant 0}$ we denote the sets of negative, positive, non-negative, non-positive  real numbers, respectively. The rotation group $SO(3)=\{R\in\R^{3\times 3}: RR^\top = I, \det R= 1 \}.$ The vector space of real $n$ by $n$ skew symmetric matrices is denoted as $\frakso(3)$. The vectors $\one_n$ and $\mathbf{0}_n$ represents a $n$-dimensional column vector with each entry being $1$ and $0$, respectively. 
We denote 
\begin{align*}
E_1 & = \diag{[-1,-1,1]} \\
E_2 & = \diag{[-1,1,-1]} \\
E_3 & = \diag{[1,-1,-1]}, 
\end{align*}
respectively.



\section{Preliminaries}\label{s:Preliminaries}

In this section, we briefly review some essentials about rigid body attitudes \cite{schaub},  and give some definitions for Filippov solutions \cite{filippov1988}.

Next lemma follows from Euler's Rotation Theorem.
\begin{lemma}
The exponential map
\begin{equation}
\exp :\frakso(3)\rightarrow SO(3)
\end{equation}
is surjective.
\end{lemma}

The tangent space at a point $R\in SO(3)$ is 
\begin{align}
T_{R}SO(3)=\{R\omega: \omega\in \frakso(3) \}.
\end{align}

For $SO(3)$, two exponential maps are needed, namely Riemannian exponential at the point $R$ and Lie group exponential, denoted $\exp_R$ and $\exp$ respectively. 

For any $p=[p_1,p_2,p_3]^\top\in\R^3$ and $\hat{p}\in\frakso(3)$ given as
\begin{equation}\label{e:basic:hat}
 \hat{p}:=
 \begin{pmatrix}
   0 &-p_3 &p_2\\
   p_3& 0 &-p_1\\
   -p_2& p_1 & 0
 \end{pmatrix},
\end{equation}
Rodrigues' formula is the right-hand side of
\begin{equation}\label{e:exp-map}
\exp(\hat{p}) = 
\begin{cases}
I_3+\frac{\sin(\|p\|)}{\|p\|}\hat{p}+\frac{1-\cos(\|p\|)}{\|p\|^2}(\hat{p})^2, & \textrm{ if } \|p\|\neq 0, \\
I_3, & \textrm{ if } \|p\|= 0.
\end{cases}
\end{equation}
The matrix $\exp(\hat{p})$ is the rotation matrix through an angle $\|p\|$ anticlockwise about the axis $p$.
The Riemannian exponential map $\exp_R:T_{R}SO(3)\rightarrow SO(3)$ is defined as 
\begin{align}
\exp_{R_1}(v) = \gamma(1)
\end{align}
where 
\begin{align}
\gamma(t) = R_1(R_1^\top R_2)t, \quad 0\leq t\leq 1
\end{align} 
is the length of the shortest geodesic curve that connect $R_1$ and $R_2$, and $\gamma'(0)=v$. The relation between these exponential maps is $\exp_R(RW)=R\exp(W)$ for any $RW\in T_{R}SO(3)$.

The principle logarithm for a matrix $R\in SO(3)$ is defined as
\begin{equation}\label{e:log-map}
\log(R) = 
\begin{cases}
\frac{\theta}{2\sin(\theta)}(R-R^\top), & \textrm{ if } \theta\neq 0,\\
\mathbf{0}, & \textrm{ if } \theta= 0
\end{cases}
\end{equation}
where $\theta=\arccos(\frac{\trace(R)-1}{2})$. We define $\log(I_3)$ as the zero matrix in $\R^{3\times 3}$. Note that \eqref{e:log-map} is not defined for $\theta=\pi$.

There are three commonly used metrics in $SO(3)$. A straightforward one is Frobenius (chordal) metric 
\begin{align}\label{e:frobenius}
d_F(R_1, R_2)&=\|R_1-R_2\|_F \\
&=\sqrt{6-\trace(R^\top_1R_2)-\trace(R^\top_2R_1)},
\end{align} 
which is Euclidean distance of the ambient space $\R^{3\times 3}$. Another metric employs the Riemannian structure, namely the Riemannian (geodesic) metric 
\begin{align}\label{e:geodesic}
d_R(R_1, R_2)=\frac{1}{\sqrt{2}}\|\log(R^{-1}_1R_2)\|_F.
\end{align}
The third one is hyperbolic metric defined as $d_H(R_1,R_2)= \|\log(R_1)-\log(R_2)\|_F$. 



One important relation between $SO(3)$ and $\R^3$ is that the open ball $B_\pi(I)$ in $SO(3)$ with radius $\pi$ around the identity, which is almost the whole $SO(3)$, is diffeomorphic to the open ball $B_{\pi}(0)$ in $\R^3$ via the logarithmic and the exponential map defined in \eqref{e:log-map} and \eqref{e:exp-map}.

\medskip

In the remainder of this section, we discuss Filippov solutions \cite{Fischer2013}. Consider the system
\begin{align}\label{e:time-varying-equation}
\dot{x} = f(x,t)
\end{align} 
where $x(t)\in\calD\subset \R^n$ denotes the state vector, $f:\calD\times [0,\infty)\rightarrow \R^n$ is Lebesgue measurable and essentially locally bounded, uniformly in $t$ and $\calD$ is an open and connected set.

\begin{definition}[Filippov solution \cite{filippov1988,Fischer2013}]
A function $x:[0,\infty)\rightarrow \R^n$ is called a solution of \eqref{e:time-varying-equation} on the interval $[0,\infty)$ if $x(t)$ is absolutely continuous and for almost all $t\in[0,\infty)$
\begin{align}\label{e:differential_inclusion}
\dot{x} \in \calF[f](x(t),t)
\end{align} 
where $\calF[f](x(t),t)$ is an upper semi-continuous, nonempty, compact and convex valued map on $\calD$, defined as
\begin{align}
\calF[f](x(t),t) := \bigcap_{\delta>0}\bigcap_{\mu(S)=0}\overline{\mathrm{co}}\big\{ f(B(x,\delta)\backslash S,t) \big\},
\label{eqn_Filippovdef}
\end{align}
where $S$ is a subset of $\R^n$, $\mu$ denotes the Lebesgue measure, $B(x,\delta)$ is the ball centered at $x$ with radius $\delta$ and $\overline{\mathrm{co}}\{\calX\}$ denotes the convex closure of a set $\calX$.
\end{definition}
If $\funcRdRd$ is continuous at $x$, then $ \calF[\funcRdRd](x)$ contains only the point $\funcRdRd(x)$.


A Filippov solution is \emph{maximal} if it cannot be extended forward in time, that is, if it is not the result of the truncation of another solution with a larger interval of definition. Next, we introduce invariant sets, which will play a key part further on. Since Filippov solutions are not necessarily unique, we need to specify two types of invariant sets. A set $\calR\subset\R^n$ is called \emph{weakly invariant} if, for each $x_0\in \calR$, at least one maximal solution of \eqref{e:differential_inclusion} with initial condition $x_0$ is contained in $\calR$. Similarly, $\calR\subset \R^n$ is called \emph{strongly invariant} if, for each $x_0\in \calR$, every maximal solution of \eqref{e:differential_inclusion} with initial condition $x_0$ is contained in $\calR$. For more details, see \cite{cortes2008, filippov1988}. We use the same definition of regular function as in \cite{Clarke1990optimization} and recall that any convex function is regular. And any $\calC^1$ continuous function is regular.


For $V:\R^n\times [0,\infty)\rightarrow\R$ locally Lipschitz in $(x,t)$, the \emph{generalized gradient} $\partial V$ is defined by
\begin{align*}
\partial V(x,t):= & \mathrm{co}\Big\{\lim_{i\rightarrow\infty} \nabla
V(x_i,t_i)\mid (x_i,t_i)\rightarrow (x,t), \\
& (x_i,t_i)\notin S\cup \Omega_{\funcRdR} \Big\},
\end{align*}
where $\nabla$ is the gradient operator, $\Omega_{\funcRdR} \subset\R^n\times [0,\infty)$ is the set of points where $V$ fails to be differentiable and $S\subset\R^n\times [0,\infty)$ is a set of measure zero that can be
arbitrarily chosen to simplify the computation, since the resulting set $\partial V(x,t)$ is independent of the choice of $S$ \cite{Clarke1990optimization}.

Given a  set-valued map $\calT:\R^n\times  [0,\infty)\rightarrow 2^{\R^n}$, the \emph{set-valued Lie derivative} $\calL_{\calT}V$ of a locally Lipschitz function $V:\R^n\times [0,\infty) \rightarrow \R$  with respect to $\calT$ at $(x,t)$ is defined as
\begin{equation}\label{e:set-valuedLie}
\begin{aligned}
\calL_{\calT}V(x,t) := & \big\{ a\in\R \mid \exists\nu\in\calT(x) \textnormal{ such that } \\
& \quad \zeta^T
\begin{bmatrix}
\nu \\ 1
\end{bmatrix}
=a,\, \forall \zeta\in \partial V(x,t)\big\}.
\end{aligned}
\end{equation}

\section{Problem formulation}\label{ss:basic_model}


In this paper we consider attitude tracking problem. The basic model can be considered as two agent network where the follower  tracks the attitude of the target. We denote the world frame as $\calF_w$, the instantaneous body frame of the target and the follower as $\calF_r$ and $\calF_1$, respectively. Let $R_r(t), R_1(t)\in SO(3)$ be the attitude of $\calF_r$ and $\calF_1$ relative to $\mathcal{F}_w$ at time $t$.

Recall that the tangent space at a point $R\in SO(3)$ is 
\begin{align}
T_{R}SO(3)=\{R\omega: \omega\in \frakso(3) \}.
\end{align}
Then the kinematics of the two attitudes are given by \cite{schaub}
\begin{equation}\label{e:basic_model_comp}
\dot{R}=\diag(R_r,R_1)\omega
\end{equation}
where
\begin{equation}
\begin{aligned}
R & =[R^\top_r,R^\top_1]^\top, \\
\omega & =[\omega^\top_r, \omega^\top_1]^\top,
\end{aligned}
\end{equation}
where $\omega_1$ is the control input to design

By asymptotic and finite time attitude tracking we mean that for the multi-agent system \eqref{e:basic_model_comp}, the absolute rotations of agent 1 track the rotation of the target in the world frame $\calF_w$ asymptotically and in finite time, respectively. In other words,
\begin{align*}
& R_1 \rightarrow R_r,  \textnormal{ as } t\rightarrow \infty, \textnormal{ and } \\
\exists T>0, \textnormal{ s.t. } & R_1 \rightarrow R_r,  \textnormal{ as } t\rightarrow T,
\end{align*}
respectively.

\section{Main result: single agent tracking}\label{s:main}

In this section, we first assume that the desired velocity $\omega_r(t)\in \frakso(3)$ and the geodesic difference  are available to the agent 1.  
Here we present two controllers as  
\begin{align}
\omega_{1,a} & =  \log(R^{-1}_1R_r) + \omega_r, \label{e:Asy-single} \\
\omega_{1,f} & = \frac{1}{\|\log(R^{-1}_1R_r)\|_F} \log(R^{-1}_1R_r) + \omega_r,\label{e:FTT-single}
\end{align}
which will be proved to achieve asymptotic and finite-time tracking, respectively.

As discontinuities are introduced if the controller \eqref{e:FTT-single} is employed, we shall understand the trajectories in the sense of Filippov, namely an absolutely continuous function $x(t)$ satisfying the differential inclusion
\begin{equation}\label{e:inclusion1}
\begin{aligned}
\begin{bmatrix}
\dot{R}_r \\ \dot{R}_1
\end{bmatrix} & 
\in 
\begin{bmatrix}
R_r \omega_r \\ \calF[R_1 \omega_{1,f}]
\end{bmatrix} \\
& =: \calF_1
\end{aligned}
\end{equation}
for almost all time, where we used Theorem 1(5) in \cite{paden1987}.

\begin{theorem}\label{thm: known-ref-geod}
Consider system \eqref{e:basic_model_comp}. Assume the system initialized without singularity, i.e., $ \arccos(\frac{\trace(R^\top_r(0)R_1(0))-1}{2}) \neq \pi$. Then 
\begin{enumerate}
\item the singularity is avoided for all time for both controller \eqref{e:Asy-single} and \eqref{e:FTT-single};
\item the attitude $R_1$ tracks $R_r$ exponentially and in finite time, respectively, by \eqref{e:Asy-single} and \eqref{e:FTT-single}. For \eqref{e:FTT-single}, the conclusion holds for all the solutions.
\end{enumerate} 
\end{theorem}

\begin{proof}
The proof is divided into two parts, one for each controller \eqref{e:Asy-single} and \eqref{e:FTT-single}.

\textsf{Part I:} In this part, we prove that by using controller \eqref{e:Asy-single}, the asymptotic tracking is achieved and the singularity is avoided. We can write the closed-loop as
\begin{equation*}
\begin{aligned}
\dot{R}_r & = R_r \omega_r \\
\dot{R}_1 & = R_1 (\log(R^{-1}_1R_r) + \omega_r)
\end{aligned}
\end{equation*}

Notice that the singularity only happens at $\theta = \arccos(\frac{\trace(R^\top_rR_1)-1}{2}) =\pi$, hence we only need to show that $\theta(t) \in [0,\pi)$ for all $t\geq 0$. Notice that 
\begin{equation}\label{e:gradient_theta}
\begin{aligned}
\frac{\partial \theta}{\partial R_r} & = \frac{-1}{\sqrt{1-\Delta^2}} \frac{\partial \Delta}{\partial R_r}  = \frac{-1}{2\sqrt{1-\Delta^2}}R_1, \\
\frac{\partial \theta}{\partial R_1} & = \frac{-1}{\sqrt{1-\Delta^2}} \frac{\partial \Delta}{\partial R_1}  = \frac{-1}{2\sqrt{1-\Delta^2}}R_r,
\end{aligned} 
\end{equation}
where $\Delta = \frac{\trace(R^\top_1R_r)-1}{2}$. Then we have 
\begin{align}
\dot{\theta}(t)  = & \trace (\frac{\partial ^\top \theta}{\partial R_r}\dot{R}_r + \frac{\partial ^\top \theta}{\partial R_1}\dot{R}_1) \\
 = & \frac{-1}{2\sqrt{1-\Delta^2}} \trace \Big(  R^\top_1 R_r \omega_r +  R^\top_r R_1 \omega_r \\ 
 &  R^\top_r R_1 \log(R^\top_1 R_r) \Big) \\
 = & \frac{-1}{2\sqrt{1-\Delta^2}} \trace \Big( R^\top_r R_1 \log(R^\top_1 R_r) \Big) \\
 = & \frac{-1}{2\sqrt{1-\Delta^2}} \frac{\theta}{\sin(\theta)} \trace \Big( I-R^\top_rR_1 R^\top_rR_1 \Big) \\
 \leq & 0
\end{align}
where the last inequality is based on the fact that $R^\top_rR_1 R^\top_rR_1\in SO(3)$. This proves that if the singularity is avoid at the initialization, then it is avoided along the trajectory.

Then consider the Lyapunov function $W(R_r,R_1)=d_R^2(R_r,R_1) = \frac{1}{2}\|\log(R^\top_r R_1) \|_F^2$, and we have 
\begin{align}
\frac{\partial W}{\partial R_r} & = - R_r \log(R^\top_rR_1) \\
\frac{\partial W}{\partial R_1} & = - R_1 \log(R^\top_1R_r)
\end{align}
and
\begin{align*}
\dot{W}(t)  = & \trace(\frac{\partial^\top W}{\partial R_r}\dot{R}_r + \frac{\partial^\top W}{\partial R_1}\dot{R}_1) \\
 = & - \trace ( \log^\top(R^\top_1R_r)\log(R^\top_1R_r) ) \\
 & + \trace \Big( \log^\top(R^\top_rR_1)\omega_r +\log^\top(R^\top_1R_r)\omega_r \Big) \\
 = & -  \trace \Big( \log^\top(R^\top_1R_r)\log(R^\top_1R_r) \Big) \\
 = & - 2 W.
\end{align*}
Hence by LaSalle-Yoshizawa Theorem (see e.g., \cite{Chen2015nonlinear}), the follower tracks the attitude of the target exponentially.

\textsf{Part II:}
In this part we prove that the finite-time tracking can be achieved by controller \eqref{e:FTT-single} and the singularity is avoided. The proof is similar to Part I. Hence we only provide the sketch.

For this case, we need to consider differential inclusion \eqref{e:inclusion1} since the discontinuity is present. Notice that the function $W$ and $\theta$ is $\calC^1$, hence regular. Then for $\theta \neq 0$, i.e., $R^\top_1R_r\neq I$, we have 
\begin{align*}
\calL_{\calF_1}\theta  = & \big\{\frac{-1}{2\sqrt{1-\Delta^2}} \frac{\theta}{\sin(\theta)} \frac{1}{\|\log(R^\top_1R_r)\|_F} \\
& \trace \Big( I-R^\top_rR_1 R^\top_rR_1 \Big) \big\} \\
\subset & \R_{-}.
\end{align*}
By the fact that $\theta$ is $\calC^1$ continuous, hence $\theta(R_r(t),R_1(t))$ is absolutely continuous and $\dot{\theta}(t)$ exists almost everywhere which belongs to $\calL_{\calF_1}\theta$. Then 
\begin{align}
\theta(t) = \int_{0}^{t} \dot{\theta}(\tau) d\tau + \theta(0) \leq \theta(0), 
\end{align}
which indicate the singularity is avoided. 
 
Next, we prove the finite-time tracking. Consider the error $V := W^\alpha$ with $\alpha > \frac{1}{2}$. Then the set-valued derivative is given as  
\begin{equation*}
\calL_{\calF_1}V  = 
\begin{cases}
\{ -\alpha\sqrt{2} V^{\beta} \}, & \textrm{ if } R_1^\top R_r \neq I \\
\{0 \}, & \textrm{ if } R_1^\top R_r = I
\end{cases} 
\end{equation*}
where $\beta = \frac{2\alpha-1}{2\alpha}\in (0,1)$. Notice that 
\begin{align}
\{ (R_r,R_1)\mid  0\in \calL_{\calF_1} V \} = \{(R_r,R_1) \mid V = 0 \},
\end{align} 
and $\dot{V}$ exists when $V\neq 0$, and $\dot{V}$ exists almost everywhere when $V=0$ (by the fact that $V$ is $\calC^1$, hence regular) and $\dot{V}\subset \calL_{\calF_1}V = \{0\}$. In other words, we have 
\begin{align}
\dot{V} = -\alpha\sqrt{2} V^{\beta}, \textnormal{ for } V\neq 0
\end{align}
with $\beta \in (0,1)$, which implies that $V$ converge to the origin in finite time (see, e.g., \cite{Haimo1986,Ghasemi2014}). Hence we the follower tracks the attitude of the target in finite time. 
   
\end{proof}

In the controller \eqref{e:Asy-single} and \eqref{e:FTT-single}, it is assumed that the geodesic state difference is available. In the rest part of this section, we show that the same conclusion as in Theorem \ref{thm: known-ref-geod} can be derived for the controller with Frobenius difference, which is relative information as well, i.e.,
\begin{align}
\omega_{1,a} & =  R^{\top}_1R_r - R^{\top}_rR_1 + \omega_r, \label{e:Asy-single_Fro} \\
\omega_{1,f} & = \frac{1}{\|R_1-R_r\|_F} \big(R^{\top}_1R_r - R^{\top}_rR_1 \big) + \omega_r,\label{e:FTT-single_Fro}.
\end{align}

\begin{corollary}\label{thm: known-ref-fro}
Consider system \eqref{e:basic_model_comp}. Assume the system initialized without singularity, i.e., $ \arccos(\frac{\trace(R^\top_r(0)R_1(0))-1}{2}) \neq \pi$. Then 
\begin{enumerate}
\item the singularity is avoided for all time for both controller \eqref{e:Asy-single_Fro} and \eqref{e:FTT-single_Fro};
\item the attitude $R_1$ tracks $R_r$ exponentially and in finite time, respectively, by \eqref{e:Asy-single_Fro} and \eqref{e:FTT-single_Fro}. For \eqref{e:FTT-single_Fro}, the conclusion holds for all the solutions.
\end{enumerate} 
\end{corollary}

\begin{proof}
Here the proof is similar to the one of Theorem \ref{thm: known-ref-geod}, hence we only provide the sketch. Here the proof is again divided into two parts.

\textsf{Part I:}
First, by \eqref{e:gradient_theta}, we have 
\begin{align}
\dot{\theta}(t)  = & \trace (\frac{\partial ^\top \theta}{\partial R_r}\dot{R}_r + \frac{\partial ^\top \theta}{\partial R_1}\dot{R}_1) \\
 = & \frac{-1}{2\sqrt{1-\Delta^2}} \trace \Big(  R^\top_1 R_r \omega_r +  R^\top_r R_1 \omega_r \\ 
 &  R^\top_r R_1 (R^\top_1 R_r- R^\top_r R_1) \Big) \\
 = & \frac{-1}{2\sqrt{1-\Delta^2}} \trace \Big( R^\top_r R_1 (R^\top_1 R_r- R^\top_r R_1) \Big) \\
 = & \frac{-1}{2\sqrt{1-\Delta^2}} \trace \Big( I-R^\top_rR_1 R^\top_rR_1 \Big) \\
 \leq & 0.
\end{align}
Hence the singularities are avoided along the trajectory, i.e., the rotation matrices $R_r(t)^\top R_1(t)\neq E_i, i=1,2,3$ if the equality does not hold for $R_r(0)^\top R_1(0)$.

Then consider the Lyapunov function $W(R_r,R_1)= \frac{1}{2} d_F^2(R_r,R_1) = 3-\trace{R^\top_r R_1}$, then
\begin{align*}
\dot{W}(t)  = & - \trace(R^\top_r\dot{R}_1 + \dot{R}^\top_rR_1) \\
 = & - \trace ( R^\top_r R_1 \omega_r + \omega^\top_r R^\top_rR_1) \\
 & - \trace \Big( R^\top_r R_1(R^\top_1 R_r-R^\top_r R_1) \Big) \\
 = & -  \trace \Big( I - R^\top_r R_1R^\top_r R_1 \Big) \\
 \leq & 0.
\end{align*}
Hence by LaSalle-Yoshizawa Theorem (see e.g., \cite{Chen2015nonlinear}), the follower tracks the attitude of the target asymptotically. Moreover, as the $\theta\rightarrow 0$ asymptotically, there exists $T$ such that for any $t\geq T$, we have 
\begin{align}
\trace(R^\top_r R_1R^\top_r R_1) \leq \trace{R^\top_r R_1}.
\end{align}
Hence for $t\leq T$, $\dot{W} \leq - W$. This implies the convergence is in fact exponential.

\textsf{Part II:}
The conclusion for controller \eqref{e:FTT-single_Fro} can be derived similar to the proof of Theorem \ref{thm: known-ref-geod}, by using the Lyapunov function $V = W^\alpha$ for $\alpha \in (\frac{1}{2},\infty)$. 
\end{proof}

\begin{remark}
For the finite-time tracking controller \eqref{e:FTT-single} and \eqref{e:FTT-single_Fro}, one closely related work is \cite{Du2011}. Compare the result here to the one in Section III in \cite{Du2011}, which assumes that the absolute attitude, the bounded velocity, the bounded acceleration of the target are available to the follower, the advantages of our controllers are that the control laws are very intuitive, that we do not assume that the desired velocity is bounded, and that only relative measurement is needed, i.e., the geodesic and Frobenius difference.
\end{remark}

\section{Simulations}\label{s:simulation}

In this section, we consider a specific trajectory of the target $R_r$ which is governed by 
\begin{align*}
\dot{R}_r = R_r \omega_r
\end{align*} 
where $\omega^\vee_r (t) = t\sin(3\mathds{1}_3t)$. Notice that the reference velocity $\omega^\vee_r (t)$ is unbounded which is more general than the assumption in \cite{Du2011}. Here we present the simulation results of the control protocols \eqref{e:Asy-single},\eqref{e:FTT-single},\eqref{e:Asy-single_Fro} and \eqref{e:FTT-single_Fro}, respectively. To make the graphical results more compact, we only show the trajectories of $R_r(i,1)$ and $R_1(i,1)$ for $i=1,2$, respectively. All the dynamical systems are initialized randomly without singularities.

For the case when geodesic differences are available, the results are shown in Fig.\ref{fig:geo}. The dashed line is the trajectories of the target $R_r$ and the solid ones are of the follower. Asymptotic and finite-time tracking are depicted in Fig.\ref{fig:as_geo} and Fig.\ref{fig:ft_geo}, using controller \eqref{e:Asy-single} and \eqref{e:FTT-single}, respectively. The similar results for \eqref{e:Asy-single_Fro} and \eqref{e:FTT-single_Fro} are shown in Fig.\ref{fig:fro}.

\begin{figure}[t!]
\centering
\begin{subfigure}[t]{0.5\textwidth}
\includegraphics[width=1\textwidth]{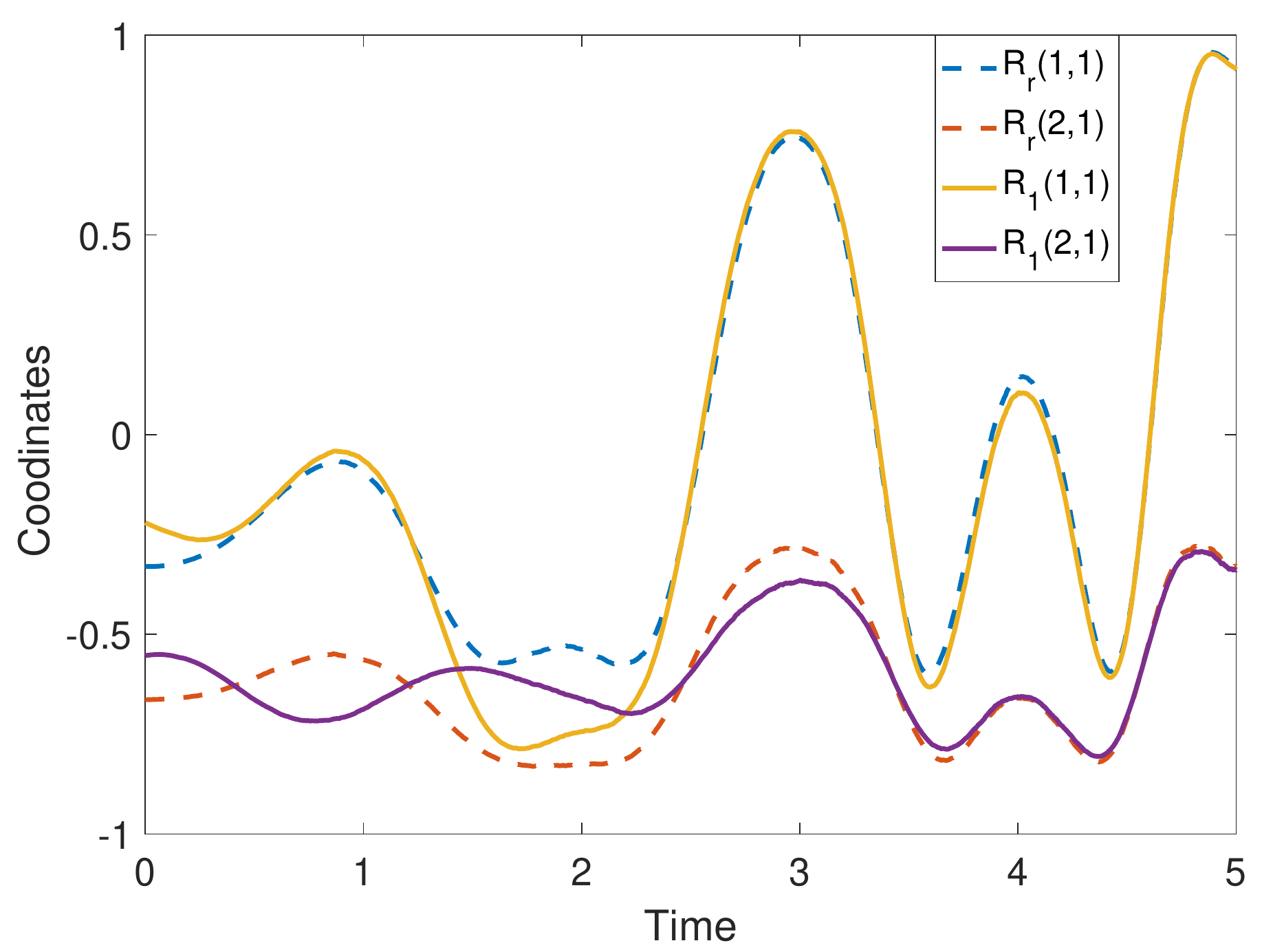}
\caption{Evolution of the coordinates of $R_r(i,1), R_1(i,1),i=1,2$ of the system controlled by \eqref{e:Asy-single}.}\label{fig:as_geo}
\end{subfigure}
~
\begin{subfigure}[t]{0.5\textwidth}
\includegraphics[width=1\textwidth]{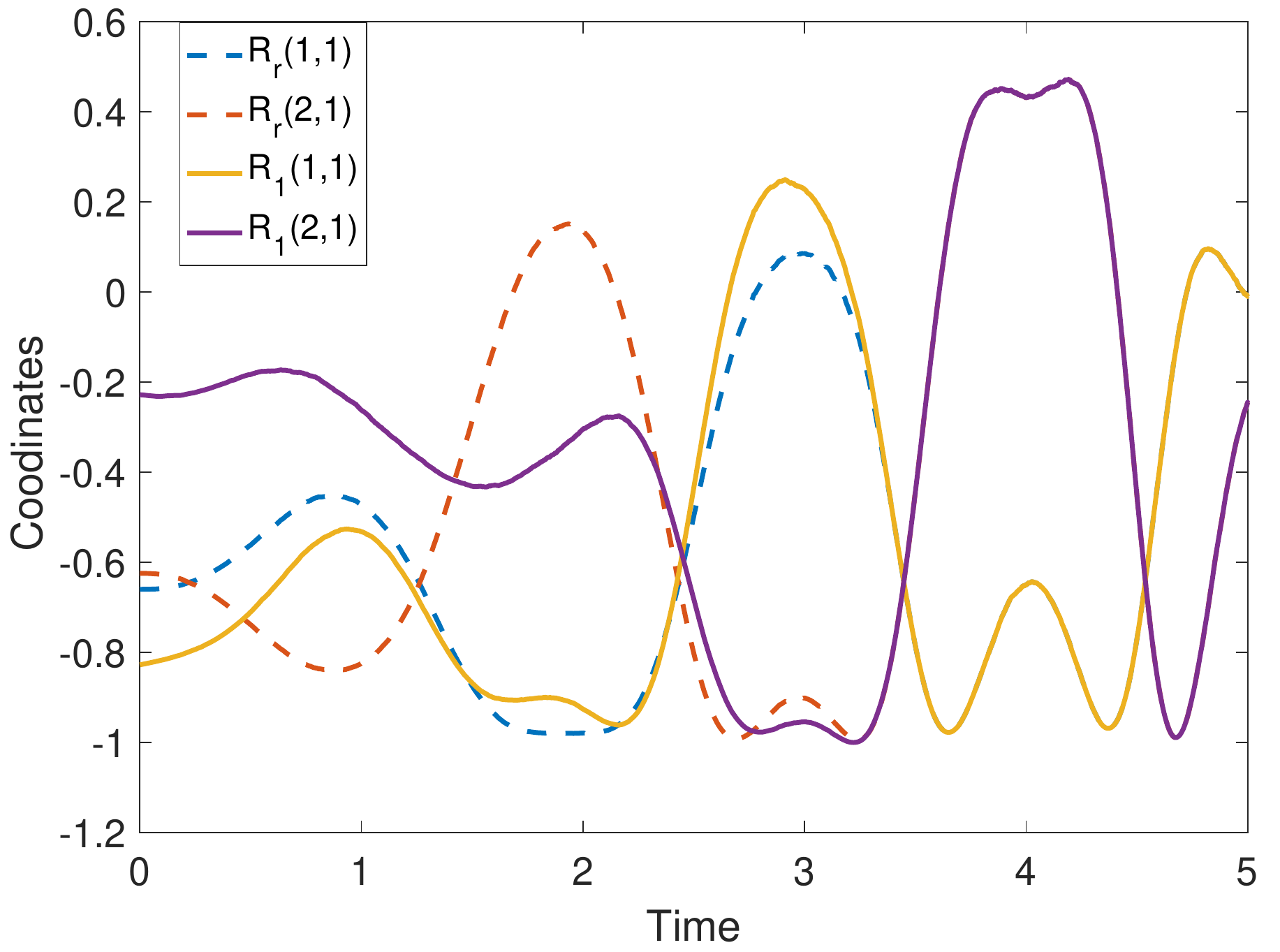}
\caption{Evolution of the coordinates of $R_r(i,1), R_1(i,1),i=1,2$ of the system controlled by \eqref{e:FTT-single}. Finite-time consensus is achieved.}\label{fig:ft_geo}
\end{subfigure}
\caption{The simulation using geodesic difference.}\label{fig:geo}
\end{figure}

\begin{figure}[t!]
\centering
\begin{subfigure}[t]{0.5\textwidth}
\includegraphics[width=1\textwidth]{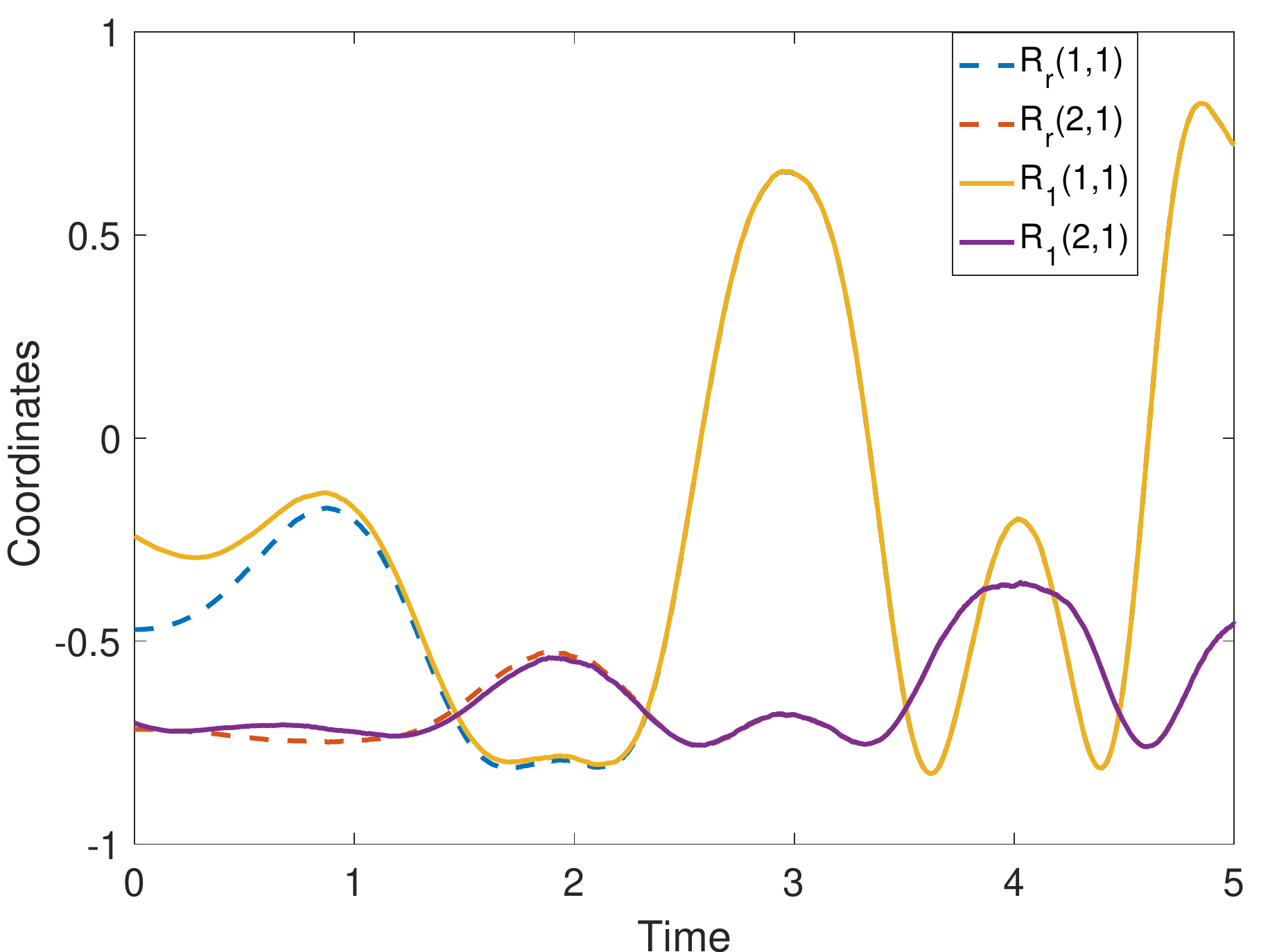}
\caption{Evolution of the coordinates of $R_r(i,1), R_1(i,1),i=1,2$ of the system controlled by \eqref{e:Asy-single_Fro}.}\label{fig:as_fro}
\end{subfigure}
~
\begin{subfigure}[t]{0.5\textwidth}
\includegraphics[width=1\textwidth]{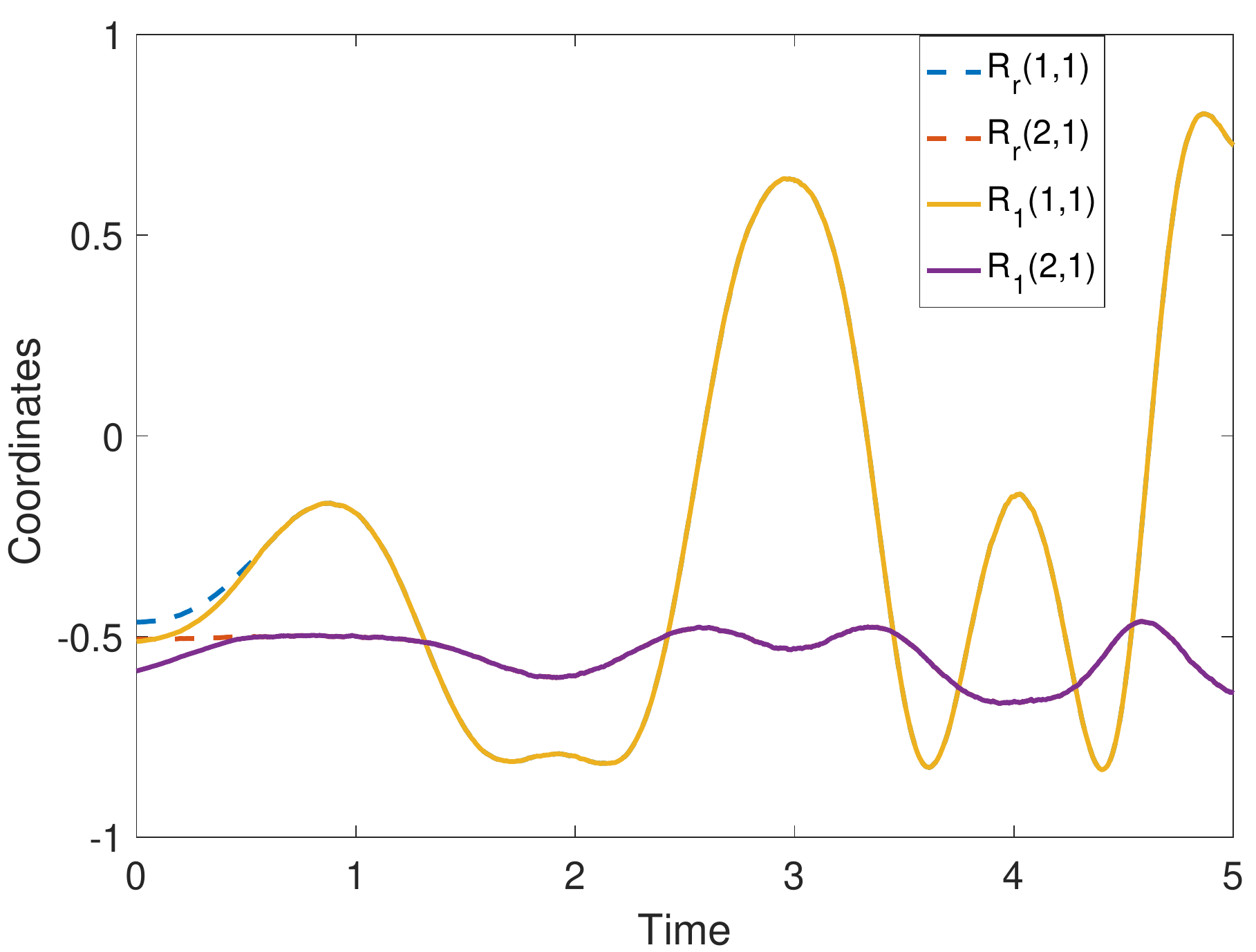}
\caption{Evolution of the coordinates of $R_r(i,1), R_1(i,1),i=1,2$ of the system controlled by \eqref{e:FTT-single_Fro}. Finite-time consensus is achieved.}\label{fig:ft_fro}
\end{subfigure}
\caption{The simulation using Frobenius difference.}\label{fig:fro}
\end{figure}

\section{Conclusion}\label{s:conclusion}

In this paper, we consider the asymptotic and finite-time attitude tracking problem. Based on the geodesic state difference, one asymptotic and finite-time tracking protocols are proposed. These protocols stabilize the system almost globally, i.e., the state of the follower tracks the attitude of the target if the system is initialized without singularity. For the finite-time controller, the solution of the closed-loop system is understood in the sense of Filippov. Similar protocols, asymptotic and finite-time one, are proposed if the Frobenius state differences are available. Future topics include estimation of the reference velocity using internal model principle, and tracking protocols using adaptive control mechanisms e.g.,prescribed performance control.



\bibliographystyle{plain} 
\bibliography{ref,C:/Users/bartb/Documents/Literature/all}

\end{document}